\documentclass[12pt,twoside]{article}
\usepackage{verbatim,amssymb,amsmath}
\usepackage{graphicx,array}
\usepackage{mystyle}
\usepackage{color}

\renewcommand{\b}{\mathfrak b}

\title{Weak Diffeomorphisms and Extremals for Scalar Conservation Laws}
\author{
{Prerona Dutta} 
\thanks{Affiliation - Department of Mathematics, Xavier University of Louisiana : pdutta@xula.edu}{, Barbara Lee Keyfitz}
\thanks{Affiliation - Department of Mathematics, The Ohio State University : keyfitz.2@osu.edu\\
}
}
\begin{document}
\bibliographystyle{abbrv}
\maketitle
%\begin{center}{\em \today}\end{center}

\begin{abstract} 
Scalar conservation laws in one space variable allow a Lagrangian (particle path) formulation. The Lagrangian trajectory in the infinite-dimensional group of diffeomorphisms on the physical space can be written as a system of conservation laws. The relation between solutions of the Cauchy problem for the conservation law and solutions of the corresponding Cauchy problem on the diffeomorphism group extends to weak solutions of the coresponding problems. The correspondence between particle paths and transport equations is analogous to that between a Lie group and the corresponding Lie algebra.

This paper establishes that for scalar conservation laws
the particle paths are extremals of an action functional on the space
of diffeomorphisms; that is, they are geodesics in some metric.
In some examples of systems of conservation laws, 
including the physical example of isentropic gas dynamics in one space 
dimension, diffeomorphism representations also exist 
and may be interpreted as extremals of action functionals.

\end{abstract}

%%%%%%%%%%%%%%%%%%%%
\section{Introduction}

This paper \footnote{2020 Mathematics Subject Classification: Primary 35L65; Secondary 35L45, 58E30} continues the exploration of
an intriguing connection between scalar hyperbolic conservation laws and elements of
an infinite-dimensional diffeomorphism group.
The idea that connections of this sort might exist originated with Vladimir Arnold in a classic paper 
published in 1966 \cite{Ar2} and has been put to use by many authors since, for example the 
celebrated 1970 result by David Ebin and Jerrold Marsden, \cite{EM}, 
which provided a new way of solving the incompressible Euler equations of a perfect fluid.
A survey by Michor \cite{Mich} puts these ideas in context.

In the case of fluid dynamics equations, the so-called Eulerian form describes  fluid behavior by 
means of equations for the density, velocity and pressure of the fluid at each point in space.
On the other hand, the position at a 
later time of a particle of fluid initially at a given position constitutes a diffeomorphism of space.
This is specified by the so-called Lagrangian equations of the fluid (though the 
underlying infinite-dimensional group of diffeomorphisms did not classically figure in the analysis).
In the work of Ebin and Marsden on incompressible fluids, the diffeomorphisms are volume-preserving, 
and the group motion that solves the initial-value problem is a geodesic curve under a right-invariant 
Riemannian metric.

Extensions of these ideas have featured in much recent research. 
Adrian Constantin and Boris Kolev \cite{CoKo} start with an abstract group of 
$C^\infty$ diffeomorphisms, ${\cal D}(M)$, 
on the real line or on the circle, and identify geodesics under various Riemannian metrics.
In the corresponding algebras these objects
satisfy partial differential or integro-differential equations.
Using the $L^2$ metric,
with the notation $\gamma(x,t)$ for elements of the diffeomorphism 
group, Constantin and Kolev find that the elements of the 
corresponding algebra, $u(x,t) =\gamma_t\circ\gamma^{-1}$, satisfy the 
quasilinear transport equation $u_t+3uu_x=0$.
Their approach determines the geodesic flow geometrically by
constructing a metric covariant derivative and a Riemannian connection
on  ${\cal D}(M)$.
They also show that the same equation results from seeking 
a length-minimizing path on the
space ${\cal D}^k(M)$ of $C^k$ diffeomorphisms, for $k\geq 2$.
To be precise, they find an extremal of the action on ${\cal D}^k(M)$ given by the right-invariant $L^2$ metric.
Implicit in their derivation is that the extremal in ${\cal D}(M)$
or ${\cal D}^k(M)$ also satisfies a partial differential equation. 
(As noted in \cite{CoKo},  ${\cal D}^k(M)$ lacks the smoothness requirements to be a Lie
group.)

With this result as a point of departure, \cite{Dutta1} and \cite{HKT} showed that initial value problems for the 
equations, both in the group of diffeomorphisms (the Lagrangian formulation) 
and in the corresponding algebra (the Eulerian formulation) are in fact Cauchy problems for
hyperbolic conservation laws.
Both possess weak solutions which exist for all positive time, and the correspondence
between the problems extends to their weak solutions\footnote{Weak solutions of conservation
laws are not unique without an admissibility condition.
With standard admissibility conditions for both formulations, this correspondence is between
admissible weak solutions for both.}.
The weak solutions in the algebra are no longer diffeomorphisms; they are bi-Lipschitz
homeomorphisms, described in those papers as `weak diffeomorphisms'.
Every Cauchy problem for a scalar conservation law in a single space variable
and a number of systems admit such a `Lagrangian' formulation.

Missing from the approach in \cite{Dutta1} and \cite{HKT} is 
the motivating feature of the work of Constantin and Kolev: 
the relation of the diffeomorphism flows to geodesics.
The purpose of this paper is to show that the diffeomorphism flows  
are extremals of an action functional defined on the algebra.
The quadratic, $L^2$, action yields the equation noted above, $u_t+3uu_x=0$.
Extending the correspondence to weak solutions determines the
correct conservation form of this equation to be 
\begin{equation} \label{correct}
(\sqrt{u})_t + ((\sqrt{u})^3)_x=0\,,
\end{equation}
which requires a sign for $u$, as explained in Section \ref{back}.

%%%%%%%%%%%%%%
\subsection{Background} \label{back}

The papers \cite{Dutta1} and \cite{HKT} cited above demonstrate a correspondence 
between weak solutions $\rho\in BV\cap L^1_{\textrm{loc}}$
of the Cauchy problem
for a scalar conservation law
\begin{equation} \label{csc}
\rho_t + f(\rho)_x=0\,,\quad \rho(x,0) = \rho_0(x)\,,\quad x\in M,
\end{equation}
and bi-Lipschitz homeomorphisms $\gamma$ recovered from a conservation law system
of the form
\begin{equation} \label{tcs}
\begin{split}  \eta_t&=F\left(\frac{v}{\eta}\right)_{\!x}\,,\quad \eta(x,0)=1,\\
			v_t&=0\,,\quad \qquad \quad v(x,0) = u_0(x)\,.
\end{split}
\end{equation}
It is noteworthy that \eqref{tcs} is a `Temple system', a particular type of conservation law system that
has large-data solutions.
The domain, $M$, of $x$ may be either $\mathbb R$ or $\mathbb S^1$.

The relation between \eqref{csc} and \eqref{tcs} is as follows.
The problem \eqref{csc} must
first undergo an affine transformation to an equivalent Cauchy problem.
\begin{proposition}[Dutta \cite{Dutta1}] \label{one}
For the Cauchy problem \eqref{csc}, there are
mappings $\rho \mapsto \rho+L$ so that $\rho_0>0$, and $f\mapsto f+K$, with $L$ and $K$
constant, 
so that the quotient $f(a)/a$ is positive and strictly increasing on the range of $\rho_0$.
Solutions to the scaled equation are equivalent to solutions of the original problem.
\end{proposition}
 To avoid introducing additional notation, assume in this paper that \eqref{csc} already has this structure.
Then the functions in \eqref{tcs} are defined from \eqref{csc} as follows:
\begin{definition} \label{two}
Given the Cauchy problem \eqref{csc}, define
\[F(\rho)=f(\rho)/\rho\equiv u;  \rho=F^{-1}(u)\equiv g(u); \textrm{ and } u_0=F(\rho_0).\]
\end{definition}
The following theorem applies.
\begin{theorem}[Dutta \cite{Dutta1}, Theorem 2.2] \label{three}
Given positive $BV$ data $u_0$ and
assuming that $F$ is positive and strictly increasing over the range of the data,
 \eqref{tcs} is a strictly hyperbolic Temple system with one degenerate
characteristic field.
The Cauchy problem \eqref{tcs} has a unique admissible weak solution 
$(\eta,v) \in ({\cal C}(0,\infty), BV^2)$ for $t>0$.
\end{theorem}
Finally, $\gamma$ is recovered from the solution of \eqref{tcs} by means of a theorem in \cite{HKT}.
\begin{theorem}[Holmes, Keyfitz, Tiglay \cite{HKT}, Theorem 4] \label{four} \label{five}
For functions\\
\noindent $(\eta,v) \in ({\cal C}(0,\infty), BV^2)$,
the distributional solution $\gamma$ to the equations
$$ \gamma_x(x,t) = \eta(x,t)\,, \quad \gamma_t(x,t)=F\left(\frac{v(x,t)}{\eta(x,t)}\right)\,,
\quad \gamma(x,0)=x$$
is well defined, absolutely continuous, and invertible, and its inverse is absolutely continuous.
Define 
$$ \rho \equiv \frac{v(\gamma^{-1}(x,t),t)}{\eta(\gamma^{-1}(x,t),t)}\,.$$
Then $\rho$ is an admissible weak solution to the Cauchy problem \eqref{csc}.
\end{theorem}
The main tool in the proof is the Radon-Nikodym theorem.

The properties of $\gamma$ were established in \cite {HKT}.
Note that  $\rho = g(u)$ and $u= \gamma_t\circ\gamma^{-1}$;
formally, we also have $\gamma_t=u\circ\gamma$. 
%%%%%%%
\begin{remark}
The relation $u=\gamma_t\circ \gamma^{-1}$ determines an element $u$ of a Lie algebra
from an element $\gamma$ in a Lie group, in this case ${\cal D}(M)$.
As noted above, our spaces and the connections between them do not satisfy the
smoothness requirements of Lie algebras and groups.
\end{remark}

%%%%%%%
Given  the initial-value problem \eqref{tcs} with positive data $u_0$, a unique
weak global solution $(\eta,v)$ is defined for all $t>0$; the bi-Lipschitz homeomorphism $\gamma$ is
similarly uniquely defined; and the correspondence
 $\gamma \mapsto \rho$ defines a unique conservation law \eqref{csc} with $f(\rho)=\rho F(\rho)$.
A forward construction, which defines $\gamma$ from $\rho$ by
$\gamma_t=u(\gamma,t)$, can be carried out directly
only for classical solutions to \eqref{csc} when the initial
data are at least Lipschitz.

\begin{remark}
The definition of $F$ allows one
to write \eqref{csc} as a transport equation $\rho_t+(u\rho)_x=0$, or $g(u)_t+\big(ug(u)\big)_x=0$.
With this as the underlying idea, $u$ is the (Eulerian) velocity of a material with density $\rho$,
and $\gamma(x,t)$ is the (Lagrangian) path of a particle with initial position $x$.
In mechanics, it would be unusual for velocity to be a function of density.
Scalar conservation laws, for the most part, do not describe any physical situation.
Intuitively, however, one may think of the Lighthill-Whitham-Richards model \cite{Whi}
for uni-directional traffic flow, $\rho_t + \big(\rho F(\rho)\big)_x=0$, where   
$u=F(\rho)$ is the velocity of traffic with density $\rho$. 
In a simplified version of the traffic flow equation, $u=1-\rho$ is an affine function of
the density, and hence is also conserved, even for weak solutions.
The simplified equation is also equivalent to Burgers' equation.
\end{remark}

Theorem \ref{four}
extends a correspondence between classical solutions of 
scalar conservation laws and diffeomorphisms to a
correspondence between weak solutions and weak diffeomorphisms.

The short section 3.4 of \cite{CoKo} 
considers the possibility of finding diffeomorphisms for times beyond the breakdown time of
$u_t+3uu_x=0$ for given Cauchy data, and concludes that this is impossible.
The current paper provides an explanation: classical solutions of one equation correspond to classical
solutions of the other; smooth solutions break down simultaneously in both systems, and are replaced by
weak solutions.
The counterexample in \cite{CoKo} fails in part 
because the correct conservation relation for $u_t+3uu_x=0$,
{\em if one is interested in the diffeomorphism that is a geodesic under the $L^2$ Riemannian metric\/} is
equation \eqref{correct}, which requires positive data.
See also the discussion at the end of Section \ref{A}.

Returning to the general case,
the objective of this paper is to show the existence of a unique action functional,
defined in equation \eqref{action} below, whose extremals lead to \eqref{tcs}.

Every Cauchy problem for a scalar conservation law (in the form
described in Proposition \ref{one}) admits a weak diffeomorphism formulation that is an extremal of an action,
and this correspondence is unique.
Conversely, every positive definite action on the space of diffeomorphisms gives rise to a 
system
of the form \eqref{tcs} and hence to a unique conservation law via
Theorem \ref{four}.
For example, 
the weak diffeomorphism for Burgers' equation, $\rho_t+ (\rho^2/2)_x=0$ with positive data
is the extremal for a cubic action potential.

To state the main result of this paper,
we define the action functional $\b$ on ${\cal D}^2(M)$, following \cite{CoKo}, by
\begin{equation} \label{action}
 \b(\gamma) = \int_0^c\int_M b\big(\gamma_t\circ\gamma^{-1}\big)\,dx\,dt\,,
\end{equation}
where $b$ is a smooth, positive, convex or concave function.
%%%%%%%%%%%%

\begin{theorem} \label{main}
If $\gamma(x,t)$ is a ${\cal C}^2$ extremal of $\b$, 
with given data $\gamma(x,0)=x$ and $\gamma_t(x,0)=u_0(x) \in {\cal C}^1$, then $\gamma$ satisfies the 
Hamilton-Jacobi equation
\begin{equation*} 
b'(\gamma_t)= \frac{b'(u_0)}{\gamma_x^2}\,,
\end{equation*}
or 
\begin{equation} \label{Beqn}
\gamma_t= B\left( \frac{b'(u_0)}{\gamma_x^2}\right)= F\left(\frac{\sqrt{b'(u_0)}}{\gamma_x}\right)\,,\quad
\gamma(x,0)=x\,.
\end{equation} 
where $B=(b')^{-1}$ and  $F= B\circ Q$; $Q$ is the function $\xi\mapsto \xi^2$.
With $\eta \equiv \gamma_x$ and $v\equiv \sqrt{b'(u_0)}$, this becomes the Temple system \eqref{tcs},
which has a weak solution for all $t>0$, and defines a bi-Lipschitz homeomorphism $\gamma$.
The weak diffeomorphism $\gamma$, in turn, corresponds uniquely to a weak solution $u$ defined by
$u\equiv \gamma_t\circ \gamma^{-1}$ of the scalar conservation law
$$\big(g(u)\big)_t+ \big(ug(u)\big)_x=0\,,$$
with $g\equiv F^{-1}$, and hence to the scalar conservation law \eqref{csc} with $\rho = g(u)$ and
$f(\rho) =\rho F(\rho)$.
\end{theorem}
%%%%%%%%%%%

The converse also holds. 
\begin{corollary} \label{coro}
Given the Cauchy problem \eqref{csc} with a flux $f$ satisfying the conditions in Proposition \ref{one},
 the diffeomorphism $\gamma$ constructed from \eqref{tcs} using Theorem
\ref{four} is an extremal of the action $\b$ with $b$ defined by $b'=B^{-1}$ where
 $B\circ Q=F$.
\end{corollary}
In particular, given a Cauchy problem for Burgers' equation with flux $f(\rho)=\rho^2/2$ and positive
data, $F=\rho/2$, $B(\xi)=\sqrt{\xi}$, and $b'=B^{-1}$; so $b'(u) = u^2$ and 
the corresponding action potential is defined by $b(y) = y^3$ (up to inessential additive and
multiplicative constants).

The rest of the paper is organized as follows. 
Section \ref{A} proves Theorem \ref{main}.
Section \ref{systs} gives some examples of systems that can also be cast in the framework of
geodesics.
We conclude this section with some comments.

\subsubsection{The Space of Weak Diffeomorphisms} \label{examp}
In \cite{Len2}, Lenells defines a space $M^{AC}$ which plays a role similar to the underlying space of
weak diffeomorphisms considered in this paper. Lenells' idea is to extend geodesics derived for the Hunter-Saxton equation to peakon solutions.
He works in the $\dot H$ (homogeneous $H^1$) metric, whereas the choice of $b(u) = \frac{u^2}{2}$ discussed after the proof of Corollary \ref{coro} corresponds to the $L^2$ metric.

As established in \cite{Dutta1} and \cite{HKT}, a `weak diffeomorphism' that provides a Lagrangian description of a given scalar conservation law in one space variable is an absolutely continuous and invertible function whose inverse is also absolutely continuous (this holds for certain systems as well). Since it is well-known (\cite{GiLe}, Theorem 17.15) that an absolutely continuous function is equal almost everywhere to a function in $W^{1,1}$, it follows that the space of weak diffeomorphisms under consideration here is a subset of $W^{1,1}$.

%%%%%%%%%%
\subsubsection{Affine Changes in the Conservation Law} 

Replacing bounded data by positive data was our first step in 
deriving a correspondence between the ``Eulerian'' and  ``Lagrangian'' forms of the Cauchy problem.
Without this preparation, the corresponding Temple system is not well-posed.
The affine changes of variables  in Proposition \ref{one} do not significantly change the nature of
the conservation law, as they amount merely to translations of the variable
and of the flux.
However, the relation between affine changes in the conservation law and 
changes in the 
diffeomorphism is more complicated.
Neither $L$ nor $K$ is uniquely determined, and the diffeomorphism equations corresponding to
different admissible choices result in different well-posed conservation law systems for $\eta$ and $v$
(or Hamilton-Jacobi equations for $\gamma$).

For the Temple system \eqref{tcs} to have global solutions, $v$ must
be non-zero, and that motivates choosing the parameter $L$ in Proposition \ref{one}.
The requirement that $F$ in Definition \ref{two} be invertible, so that the original
conservation law can be recovered from $u$ (that is, from $\gamma$), motivates the
choice of $K$ in that proposition.
Thus a scalar conservation law, \eqref{csc}, gives rise to a two-parameter family of correspondences,
between 
$$ \big(\rho+L\big)_t + \big(f(\rho+L) +K\big)_x=0\,,$$
(where the parameter $K$ enters only in the choice of $F$) and the system \eqref{tcs}
that defines $\gamma$. In this correspondence, $F$ and $u_0$ depend on the choices of $K$ and $L$.

The viewpoint in \cite{Dutta1} and \cite{HKT}, adopted here, is to regard
the conservation law \eqref{csc} as the fundamental object and the Temple system \eqref{tcs} as a device
for determining the corresponding diffeomorphism.
However, one can proceed in the other direction.
We state this as a proposition.
\begin{proposition}
Given a  system of the form \eqref{tcs} with $u_0>0$ and bounded, and $0<F(\cdot)\in C^2$ 
strictly increasing on the range of $u_0$, the corresponding scalar conservation law Cauchy problem
is 
$$ \rho_t + f(\rho)_x=0\,,\quad \rho(x,t) = \rho_0(x)\,, $$
with $\rho$ defined in Theorem \ref{four}, $f(\rho) = \rho F(\rho)$ and $\rho_0 = F^{-1}(u_0)$.
\end{proposition}

A more precise description of the relation between Eulerian
and Lagrangian versions of a scalar conservation law is that beginning with \eqref{csc}, if $\rho_0$
is bounded and positive and $f(\rho)/\rho$ is strictly increasing, there is a direct correspondence 
between $\gamma$ and $\rho$ as given in Theorem \ref{four}.
For a Cauchy problem \eqref{csc} where one or both of those conditions is not satisfied, there is a
two-parameter family of affinely equivalent Cauchy problems, each uniquely associated with a
diffeomorphism formulation, but the original problem has only this indirect correspondence with a path
in the group of diffeomorphisms.

%%%%%%%%%%%%%%%%
\section{Proof of Theorem \ref{main}} \label{A}

\begin{proof}
Consider an action functional $\b$;
\begin{equation} \label{baction}
 \b(\gamma) = \int_0^c\int_M b\big(\gamma_t\circ\gamma^{-1}\big)\,dx\,dt\,.
\end{equation}
A change of variables in the integral eliminates the need to
calculate derivatives of the inverse function.
Let $x\mapsto y=\gamma(x,t)$; then
$$  \b(\gamma) =\int_0^c\int_M b(\gamma_t)\gamma_x
\,dx\,dt\,.$$
The variational integral is 
$$ 
 \b(\gamma+\veps\zeta) = 
\int_0^c\int_M b(\gamma_t+\veps\zeta_t)(\gamma_x+\veps\zeta_x)\,dx\,dt\,.
$$
Then
\begin{align*} \frac{d}{d\veps} \b\big|_{\veps=0} &=  \frac{d}{d\veps}
\int_0^c\int_M b(\gamma_t+\veps\zeta_t)(\gamma_x+\veps\zeta_x)\,dx\,dt\,\bigg|_{\veps=0}\\
&= \int_0^c\int_M b'(\gamma_t)\zeta_t\gamma_x + b(\gamma_t) \zeta_x\,dx\,dt\\
&=- \int_0^c\int_M\zeta\left[ \gamma_x b''(\gamma_t)\gamma_{tt} +2 b'(\gamma_t) 
\gamma_{tx}\right]\,dx\,dt\,,\\
&=-\int_0^c\int_M \zeta b'(\gamma_t)\gamma_x
\left[\frac{b''(\gamma_t)}{b'(\gamma_t)}\gamma_{tt} 
+ 2 \frac{\gamma_{tx}}{\gamma_x}\right]\,dx\,dt \\
&= 0\,, \label{ahaeqn}
\end{align*}
where integrations by parts are carried out with respect to $t$ and $x$, noting that there are no contributions
from the boundaries.
This relation holds for all choices of $\zeta$, and so
the product $\zeta b' \gamma_x$ is an arbitrary function, since $b'\neq 0$ and $\gamma_x\neq 0$.
We obtain the equation
\begin{equation} \label{ahaeqn}
 \frac{b''(\gamma_t)}{b'(\gamma_t)} \gamma_{tt} + \frac{2\gamma_{tx}}{\gamma_x} 
=  \frac{\partial}{\partial t} \big[\log\big(b'(\gamma_t)\gamma_x^2\big)\big] = 0\,,
 \end{equation}
so $b'(\gamma_t)\gamma_x^2$ is constant.
Since $\gamma_t(x,0)=u_0(x)$ and $\gamma_x(x,0)=1$,
we have
\begin{equation*} 
b'(\gamma_t)= \frac{b'(u_0)}{\gamma_x^2}\,.
\end{equation*}
As $b'$ is invertible, this yields a Hamilton-Jacobi equation for $\gamma$:
$$ \gamma_t = B\left(\frac{b'(u_0)}{\gamma_x^2}\right)\,,$$
where $B^{-1}= b'$.
Comparing this to \eqref{tcs} shows that  $\gamma$ is the diffeomorphism
form of a scalar conservation law $\rho_t+ (\rho F(\rho))_x=0$ with the function $F$ given by 
$F(\xi) = B(\xi^2)$.\\
\end{proof}

\begin{proof}[{\sc{of Corollary}} \ref{coro}]
Defining $u$ by $u=\gamma_t\circ\gamma^{-1}$ and calculating derivatives of $\gamma_t=u\circ\gamma$,
we find
$$ \gamma_{tt} = (uu_x+u_t)\circ \gamma\quad \textrm{and}\quad \gamma_{tx} =
\gamma_x (u_x\circ\gamma)\,,$$ 
and so \eqref{ahaeqn} (after dropping the composition with $\gamma$) becomes
$$ u_t+ \left( u+ \frac{2b'(u)}{b''(u)}\right) u_x=0\,.$$
In the notation of Definition \ref{two}, with $u=F(\rho)$ and $\rho = g(u)$,
the conservation law for $\rho$ is $g_t+(\rho g)_x=0$, which,
written as a transport equation for $u$, is
$$ u_t+\left(u+\frac{g(u)}{g'(u)}\right)u_x=0\,.$$
Comparing these, we have 
\begin{equation} \label{ddefb}
\frac{b''(u)}{2 b'(u)} = \frac{g'(u)}{g(u)}\,,\quad \textrm{or} \quad \log b'(u) = 2 \log g(u) \quad
\textrm{or} \quad b'(u) = \big(g(u)\big)^2\,,
\end{equation}
and so $b$ is an antiderivative of $g^2$.\\
\end{proof}

The choice $b(u)=u^2/2$ is the case considered by Constantin and Kolev, and yields the
transport equation $u_t+3uu_x=0$.
In this case, $b'(u)=u$ and so $g(u) = \sqrt{u}$; $F=g^{-1}$, so $F(\rho) = \rho^2$, and the
conservation law is $\rho_t + (\rho^3)_x=0$.
Expressed in terms of the function $u$, this is $\sqrt{u}_t+ (u^{3/2})_x=0$.
For the conservation law to be valid, $u$ is necessarily non-negative.

The example of Burgers' equation, $\rho_t+(\rho^2/2)_x=0$, gives
$u=\rho/2$, so $g(u) = 2u$, and  $b(u)=4 u^3/3$.
This corresponds to an $L^p$ metric with $p=3/2$.
%%%%%%%%%%%
%%%%%%%%%%%

%%%%%%%%%%%%%
\section{Some Examples for Systems}
\label{systs}
We present a pair of positive results for systems of two conservation laws in a single space variable.
\subsection{Isentropic Gas Dynamics}
The diffeomorphism equation obtained in \cite{HKT} for the isentropic gas dynamics system,
\begin{equation*}
 \rho_t+ (u\rho)_x =0\,,\quad 
(u\rho)_t+ \big(\rho u^2/2+p(\rho)\big)_x  =0 \,,
\end{equation*}
can be written as a second-order conservation equation
for $\gamma_t$, $\gamma_x$ and $\rho_0$ (the three variables defining the state of the system):
\begin{equation} \label{systeq} 
\big(\rho_0\gamma\big)_{tt} +\left(p\left(\frac{\rho_0}{\gamma_x}\right)\right)_x=0\,.
\end{equation}
Carrying out the differentiation in \eqref{systeq} gives
\begin{equation} \label{peqn} 
\rho_0 \gamma_{tt} - \left(p'\left(\frac{\rho_0}{\gamma_x}\right)\frac{\rho_0}{\gamma_x^2}\right) \gamma_{xx}
+ p'\left(\frac{\rho_0}{\gamma_x}\right)\frac{\rho_0'}{\gamma_x} =0\,.
\end{equation}
Equation  \eqref{peqn} can be obtained as an extremal
of an action
\begin{equation} \label{b2def} 
\b(\gamma; \rho_0) = \int_0^c\int_M b\big(\gamma_t\circ\gamma^{-1}, \gamma_x\circ\gamma^{-1},
\rho_0\circ\gamma^{-1}\big)\,dx\,dt\,
\end{equation}
with  $b= b(r,s,q)$.
(Note that equation (2.6) in \cite{HKT} states that
$ \rho = (\rho_0/\gamma_x) \circ \gamma^{-1}$.)
Apply the change of variables $y=\gamma^{-1}(x,t)$, 
 introducing a factor $\gamma_x$ throughout.
Then, with
$$ \b(\gamma; \rho_0) = \int_0^c\int_M b(\gamma_t, \gamma_x,\rho_0)\gamma_x\,dx\,dt\,,
$$
 we calculate
$$ \frac{d}{d\veps}  \int_0^c\int_M b(\gamma_t+\veps\zeta_t, \gamma_x+\veps\zeta_x,\rho_0)
(\gamma_x+\veps\zeta_x)\,dx\,dt\bigg|_{\veps=0}\,.
$$
The variational equation is
\begin{align*} 0=&
\iint b_r(\gamma_t,\gamma_x,\rho_0)\zeta_t\gamma_x+ b_s(\gamma_t,\gamma_x,\rho_0))\zeta_x\gamma_x
	+b(\gamma_t,\gamma_x,\rho_0)\zeta_x \,dx\,dt\\
	=&-\iint \zeta\big[ \big(b_r(\gamma_t,\gamma_x,\rho_0)\gamma_x\big)_t +
	\big( b_s(\gamma_t,\gamma_x,\rho_0)\gamma_x +b(\gamma_t,\gamma_x,\rho_0)\big)_x \big]\,dx\,dt\\
	=&-\iint \zeta\big[\gamma_xb_{rr}\gamma_{tt}+2\big(\gamma_xb_{rs} +b_r   \big)\gamma_{xt} 
	+ \big(\gamma_xb_{ss}	+2b_s\big)\gamma_{xx}+\rho_0'\big(b_{sq}\gamma_x + b_q\big)	\big]\,dx\,dt\,.
\end{align*}	
The result (divided by $\gamma_x$) is 
\begin{equation} \label{aha}
b_{rr}\gamma_{tt}+2\left(b_{rs} +\frac{b_r}{\gamma_x}  \right)\gamma_{xt} 
	+ \left(b_{ss}	+\frac{2b_s}{\gamma_x}\right)\gamma_{xx}+ \rho_0'\left(b_{sq}+
	 \frac{b_q}{\gamma_x}\right)	=0\,.
\end{equation}
The matching of coefficients between \eqref{peqn} and \eqref{aha} is a matter of algebra unrelated to
the functional behavior of $\gamma(x,t)$, so one can calculate the form of $b$,
replacing $\gamma_t$ by $r$, $\gamma_x$ by $s$ and $\rho_0$ by $q$.
We omit the straightforward calculation.
It yields the result, easily checked: 
$$ b(\gamma_t,\gamma_x,\rho_0) = \frac{\rho_0(x) \gamma_t^2}{2\gamma_x} 
+\frac{1}{\gamma_x}\int^{\gamma_x} p\left(\frac{\rho_0(x)}{\sigma}\right)\,d\sigma $$
for the action function.

%%%%%%%%%%%%%%
\subsection{The One-Dimensional Nonlinear Wave Equation}
The nonlinear wave equation in one space dimension is
\begin{equation}
\label{nlwe} 
\rho_{tt} = (p'(\rho)\rho_x)_x \,.
\end{equation}
A diffeomorphism  corresponding to \eqref{nlwe} is found by
assigning a velocity variable.
We may proceed by
writing the equation as a first-order system:
\begin{align} \label{firsteq}
 \rho_t+ (u\rho)_x &=0\\
(u\rho)_t+ p(\rho)_x & =0 \,,\label{secondeq}
\end{align}
and continuing as in \cite{HKT} for the isentropic gas dynamics system.
Assuming classical solutions, define $\gamma$ by
\begin{equation} \label{newgammadef} \gamma_t = u(\gamma(x,t),t) \equiv u\circ\gamma\,,
\end{equation}
and define $\zeta= \rho\circ\gamma$, so \eqref{firsteq} becomes
$$ \zeta_t= (\rho_t+ u\rho_x)\circ\gamma = -(\rho u_x)\circ\gamma\,.
$$
Differentiate \eqref{newgammadef} with respect to $x$ to obtain 
$\gamma_{tx} = u_x\circ\gamma\cdot \gamma_x$, so $u_x\circ\gamma = \gamma_{tx}/\gamma_x$
and therefore 
$$ \zeta_t = -\zeta \frac{\gamma_{xt}}{\gamma_x}\,.$$
Integrating and applying the initial data gives
$$ \rho = \frac{\rho_0}{\gamma_x}\circ\gamma^{-1}\,.$$
(This works because the first equation is the same as for the isentropic gas dynamics system.)
To convert \eqref{secondeq}, we calculate the terms $(u \rho)_t$ and $p(\rho)_x$,
obtaining 
\begin{equation} \label{difftemp}-\left(\frac{\gamma_t}{\gamma_x}\right)
\cdot\left(\frac{\gamma_t\rho_0}{\gamma_x}\right)_x
+\left(\frac{\gamma_t\rho_0}{\gamma_x}\right)_t
+\frac{1}{\gamma_x}\left[p\left(\frac{\rho_0}{\gamma_x}\right)\right]_x=0\,.
\end{equation}
Carrying out the differentiations, multiplying by $\gamma_x^2$ and dividing by $\rho_0$ gives
\begin{equation} \label{nlwediffeo}
\gamma_x\gamma_{tt}
-2\gamma_t\gamma_{tx}
-\frac{1}{\gamma_x}\left[p'\left(\frac{\rho_0}{\gamma_x}\right)-\gamma_t^2\right]\gamma_{xx}
+\frac{\rho_0'}{\rho_0}\left[p'\left(\frac{\rho_0}{\gamma_x}\right)-\gamma_t^2\right]=0~.
\end{equation}
This is again a quasilinear equation, and it is easily checked that it is hyperbolic.

\begin{proposition}
Equation \eqref{nlwediffeo} is hyperbolic.
\end{proposition}
\begin{proof}
The principal part has the form $A\gamma_{tt} + 2B\gamma_{xt} + C\gamma_{xx}$, and the equation is
strictly hyperbolic if $AC -B^2<0$.
A calculation gives
$$ AC-B^2 = -\left[p'\left(\frac{\rho_0}{\gamma_x}\right)-\gamma_t^2\right]-\gamma_t^2
= -p'\left(\frac{\rho_0}{\gamma_x}\right)\,,
$$
which is negative under the standard assumption that $p$ is an increasing function.
\end{proof}
An important consideration is whether \eqref{nlwediffeo} can be put in conservation form. 
If not, then the correspondence between classical solutions of the system cannot be extended
to a correspondence between weak solutions.
To examine this, go back to \eqref{difftemp} and multiply by $\gamma_x$.
\begin{proposition}
Equation \eqref{difftemp} has the conservation form
$$ (\rho_0\gamma_t)_t  +\left( p\left(\frac{\rho_0}{\gamma_x}\right)     
  - \frac{\rho_0\gamma_t^2}{\gamma_x}\right)_x=0\,.
$$
\end{proposition}
We also note that this can be written as a first-order system, similar to the 
system in \cite{HKT} that describes isentropic gas dynamics.
\begin{proposition}
With the definitions $\eta = \gamma_x$, $w=\rho_0\gamma_t$ and $v=\rho_0$, the diffeomorphism
formulation takes the form
\begin{align*} \eta_t  - \left(\frac{w}{v}\right)_x&=0\\
w_t + \left(p\left(\frac{v}{\eta}\right) - \frac{w^2}{\eta v}\right)_x &=0\\
v_t& = 0\,.
\end{align*}
\end{proposition}
Determining a function $b(\gamma_t,\gamma_x, \rho_0)$ for which \eqref{nlwediffeo} defines an
extremum is carried out in the same way as for the previous example and yields
$$ b(r,s,q) = \frac{1}{2}\left(\frac{rq}{s}\right)^2    -P\left(\frac{q}{s}\right)\,,\quad \textrm{or} \quad
b(\gamma_t,\gamma_x,\rho_0) =\frac{1}{2}\left(\frac{\rho_0\gamma_t}{\gamma_x}\right)^2  
  -P\left(\frac{\rho_0}{\gamma_x}\right)\,.
$$
where $P'(\cdot) = p(\cdot)$.
Since there was some degree of choice in the selection of the variable $u$ that determined $\gamma$,
there may be other actions $\b$ that lead back to the same nonlinear wave equation.
%%%%%%%%%%%%%
%%%%%%%%%%%%%%%

\bibliography{diffeo}
\end{document}